\documentclass[11pt,fleqn]{article}
\usepackage{amssymb}
\usepackage{amsfonts}
\usepackage{amsmath}
\usepackage{amsthm}
\usepackage{graphicx}
\usepackage{epsfig}
\usepackage{psfrag}
\usepackage{color}
\usepackage{dsfont}
\makeatletter
\xdef\@endgadget#1{{\unskip\nobreak\hfil\penalty50\hskip1em\hbox{}\nobreak
    \hfil#1\parfillskip=0pt\finalhyphendemerits=0\par}}
\def\@qedsymbol{${}_\blacksquare$}
\def\qed{\@endgadget{\@qedsymbol}}
\newtheorem{lemma}{Lemma}[section]
\newtheorem{theorem}[lemma]{Theorem}

\newtheorem{example}[lemma]{Example}
\newtheorem{definition}[lemma]{Definition}

\newtheorem{proposition}[lemma]{Proposition}
\newtheorem{remark}[lemma]{Remark}
\newcommand{\mR}{\mathbb{R}}

\newcommand{\ith}{i^{\tiny{\text{th}}}}
\newcommand{\jth}{j^{\tiny{\text{th}}}}

\newcommand{\Exp}{\mathrm{Exp}}
\newcommand{\Ln}{\mathrm{Ln}}
\newcommand{\mL}{\mathcal{L}}
\DeclareMathOperator{\im}{im} 
\DeclareMathOperator{\rank}{rank}

\DeclareMathOperator{\spa}{span}

\def\BibTeX{{\rm B\kern-.05em{\sc i\kern-.025em b}\kern-.08em
    T\kern-.1667em\lower.7ex\hbox{E}\kern-.125emX}}
    
\title{\LARGE \bf A Graph-Theoretical Approach for \\ the Analysis and Model Reduction of \\Complex-Balanced Chemical Reaction Networks
}

\author{Shodhan Rao\thanks{Center for Systems Biology, University of Groningen, email: {\tt \small s.rao@umcg.nl}} \and Arjan van der Schaft\thanks{Johann Bernoulli Institute for Mathematics and Computer Science, University of Groningen, e-mail: {\tt\small A.J.van.der.Schaft@rug.nl}}\and Bayu Jayawardhana\thanks{Discrete Technology and Production Automation, University of Groningen, email: {\tt \small b.jayawardhana@rug.nl}} 
}
\begin{document}

\maketitle
\thispagestyle{empty}
\pagestyle{empty}

\begin{abstract}
In this paper we derive a compact mathematical formulation describing the dynamics of chemical reaction networks that are complex-balanced and are governed by mass action kinetics. %Aspects from graph theory and stoichiometry of the network are made use of in deriving our formulation. 
The formulation is based on the graph of (substrate and product) complexes and the stoichiometric information of these complexes, and crucially uses a balanced weighted Laplacian matrix.
It is shown that this formulation leads to elegant methods for characterizing the space of all equilibria for complex-balanced networks and for deriving stability properties of such networks. We propose a method for model reduction of complex-balanced networks, which is similar to the Kron reduction method for electrical networks and involves the computation of Schur complements of the balanced weighted Laplacian matrix. \\

\end{abstract}

\maketitle

\section{Introduction}
The analysis and control of large-scale chemical reaction networks poses a main challenge to bio-engineering and systems biology. Chemical reaction networks involving several hundreds of chemical species and reactions are common in living cells \cite{Palsson}. The complexity of their dynamics is further increased by the fact that the chemical reaction rates are intrinsically {\it nonlinear}. In particular mass action kinetics, which is the most basic form of expressing the reaction rates, corresponds to differential equations which are polynomial in the concentrations of the chemical species. Chemical reaction network dynamics thus are a prime example of complex networked dynamical systems, and there is a clear need to develop a mathematical framework for handling their complexity. 

One of the issues in formalizing complex chemical reaction network dynamics is the fact that their graph representation is not immediate; due to the fact that chemical reactions usually involve more than one substrate chemical species and more than one product chemical species. This problem is resolved by associating the {\it complexes} of the reactions, i.e. the left-hand (substrate) and right-hand (product) sides of each reaction, with the vertices of a graph, and the reactions with the edges\footnote{This approach is originating in the work of Horn and Jackson \cite{HornJackson}; see also Othmer \cite{Othmer} and \cite{AngeliEJC} for nice expos\'es and additional insights.}. The resulting directed graph, called the {\it graph of complexes}, is characterized by its incidence matrix $B$. Furthermore the {\it stoichiometric matrix} $S$ of the chemical reaction network, expressing the basic balance laws of the reactions, can be factorized as $S=ZB$, with the {\it complex-stoichiometric matrix} $Z$ encoding the expressions of the complexes in the various chemical species. We note that an alternative graph formulation of chemical reaction networks is the {\it species-reaction graph} \cite{Craciun, Angeli2010, AngeliEJC, Angeli2011}, which is a bipartite graph with one part of the vertices corresponding to the species and the remaining part to the reactions, and the edges expressing the involvement of the species in the reactions.

Using the graph of complexes formalism, we have developed in \cite{Ours, OursACC} a compact mathematical formulation for a class of mass-action kinetics chemical reaction networks, which is characterized by the assumption of the existence of an equilibrium for the reaction rates; a so-called thermodynamic equilibrium. This corresponds to the thermodynamically justified assumption of microscopic reversibility, with the resulting conditions on the parameters of the mass action kinetics usually referred to as the Wegscheider conditions. The resulting class of mass action kinetics reaction networks are called {\it (detailed)-balanced}. A main feature of the formulation of \cite{Ours} is the fact that the dynamics of a detailed-balanced chemical reaction network is completely specified by a {\it symmetric weighted Laplacian matrix}, defined by the graph of complexes and the equilibrium constants, together with an energy function, which is subsequently used for the stability analysis of the network. In particular, the resulting dynamics is shown \cite{OursACC} to bear close similarity with {\it consensus algorithms} for symmetric multi-agent systems. (In fact, it is shown in \cite{Ours} that the so-called {\it complex-affinities} asymptotically reach consensus.) Furthermore, as shown in \cite{OursMTNS}, the framework can be readily extended from mass action kinetics to (reversible) Michaelis-Menten reaction rates.

On the other hand, the assumption of existence of a thermodynamical equilibrium requires {\it reversibility} of all the reactions of the network, while there are quite a few well-known irreversible chemical reaction network models, including the McKeithan network to be explained shortly afterwards. Motivated by such examples we will extend in this paper the results of \cite{Ours} by considering the substantially larger class of {\it complex-balanced} reaction networks. A chemical reaction network is called \emph{complex-balanced} if there exists a vector of species concentrations at which the combined rate of outgoing reactions from any complex is equal to the combined rate of incoming reactions to the complex, or in other words each of the complexes involved in the network is at equilibrium. The notion of complex-balanced networks was first introduced in \cite{HornJackson} and studied in detail in \cite{Feinberg2, Horn, Toric, Siegel, Dick}. These systems have also been called as \emph{toric dynamical systems} in the literature (see \cite{Toric}). 

An example of a complex-balanced network is the model of T-cell interactions due to \cite{McKeithan} (see also \cite{Sontag}) depicted in Figure \ref{fig:McKeithan}. 
\begin{figure}[h]
\centerline{
  \scalebox{0.8}{
     \ifx\JPicScale\undefined\def\JPicScale{1}\fi
\unitlength \JPicScale mm
\begin{picture}(136,24)(0,0)
\put(0,0){\makebox(0,0)[cc]{$T+M$}}

\linethickness{0.3mm}
\put(6,0){\line(1,0){11}}
\put(17,0){\vector(1,0){0.12}}
\put(20,0){\makebox(0,0)[cc]{$C_0$}}

\linethickness{0.3mm}
\put(25,0){\line(1,0){12}}
\put(37,0){\vector(1,0){0.12}}
\put(40,0){\makebox(0,0)[cc]{$C_1$}}

\linethickness{0.3mm}
\put(47,0){\line(1,0){12}}
\put(59,0){\vector(1,0){0.12}}
\put(64,0){\makebox(0,0)[cc]{$\ldots$}}

\linethickness{0.3mm}
\put(70,0){\line(1,0){12}}
\put(82,0){\vector(1,0){0.12}}
\put(87,0){\makebox(0,0)[cc]{$C_i$}}

\linethickness{0.3mm}
\put(93,0){\line(1,0){12}}
\put(105,0){\vector(1,0){0.12}}
\put(111,0){\makebox(0,0)[cc]{$\ldots$}}

\linethickness{0.3mm}
\put(119,0){\line(1,0){12}}
\put(131,0){\vector(1,0){0.12}}
\put(136,0){\makebox(0,0)[cc]{$C_N$}}

\linethickness{0.3mm}
\qbezier(19,2)(14.32,5.14)(11.07,5.98)
\qbezier(11.07,5.98)(7.82,6.82)(5.5,5.5)
\qbezier(5.5,5.5)(3.15,4.2)(2.07,3.59)
\qbezier(2.07,3.59)(0.99,2.99)(1,3)
\put(9,-3){\makebox(0,0)[cc]{$k_{p,0}$}}

\put(10,3){\makebox(0,0)[cc]{$k_{-1,0}$}}

\linethickness{0.3mm}
\multiput(1,3)(0.12,0.25){8}{\line(0,1){0.25}}
\linethickness{0.3mm}
\multiput(1,3)(0.25,-0.12){8}{\line(1,0){0.25}}
\linethickness{0.3mm}
\qbezier(39,2)(29.14,7.24)(20.24,7.72)
\qbezier(20.24,7.72)(11.33,8.21)(2,4)
\linethickness{0.3mm}
\qbezier(85,2)(57.96,9.34)(37.99,9.82)
\qbezier(37.99,9.82)(18.02,10.31)(2,4)
\linethickness{0.3mm}
\qbezier(134,3)(85.59,14.02)(53.82,14.26)
\qbezier(53.82,14.26)(22.06,14.5)(2,4)
\put(26,4){\makebox(0,0)[cc]{$k_{-1,1}$}}

\put(54,6){\makebox(0,0)[cc]{$k_{-1,i}$}}

\put(96,7){\makebox(0,0)[cc]{$k_{-1,N}$}}

\put(31,-3){\makebox(0,0)[cc]{$k_{p,1}$}}

\put(52,-3){\makebox(0,0)[cc]{$k_{p,2}$}}

\put(76,-3){\makebox(0,0)[cc]{$k_{p,i}$}}

\put(98,-3){\makebox(0,0)[cc]{$k_{p,i+1}$}}

\put(125,-3){\makebox(0,0)[cc]{$k_{p,N}$}}

\end{picture}
     }
   }
\caption{McKeithan's network}
\label{fig:McKeithan}
\end{figure}
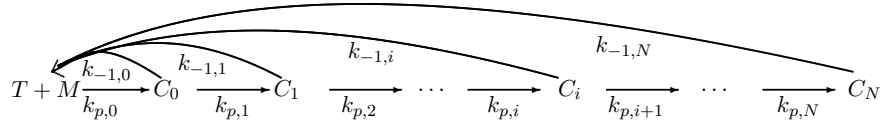
This chemical reaction network model arises in immunology and was proposed by McKeithan in order to explain the selectivity of T-cell interactions. With reference to Figure \ref{fig:McKeithan}, $T$ and $M$ represent a T-cell receptor and a peptide-major histocompatibility complex (MHC) respectively and $T+M$ is a complex for the network. For $i=1,\ldots,N$, $C_i$ represent various intermediate complexes in the phosphorylation and other intermediate modifications of the T-cell receptor $T$; $k_{p,i}$ represents the rate constant of the $\ith$ step of the phosphorylation and $k_{-1,i}$ is the dissociation rate of the $\ith$ complex. In the following, we denote by $[A]$ the concentration of a species $A$ participating in a chemical reaction network. The governing law of the reaction network is the law of mass action kinetics. This leads to the following set of differential equations describing the rate of change of concentrations of various species involved in the network:
\begin{eqnarray}\label{eq:McK}
\nonumber \frac{d[T]}{dt} &=& \frac{d[M]}{dt}=-k_{p,0}[T][M]+\sum_{i=0}^{N}k_{-1,i}[C_i] \\
\nonumber \frac{d[C_0]}{dt} &=& k_{p,0}[T][M] -(k_{-1,0}+k_{p,1})[C_0] \\
\nonumber   &\vdots & \\
\nonumber \frac{d[C_{i}]}{dt} &=& k_{p,i}[C_{i-1}] -(k_{-1,i}+k_{p,i+1})[C_{i}] \\
\nonumber   &\vdots & \\
\frac{d[C_N]}{dt} &=& k_{p,N}[C_{N-1}] -k_{-1,N}[C_N] 
\end{eqnarray} 
Observe that if the left hand side of each of the above equations is set to zero, all the concentrations $[C_i]$ for $i=0,\ldots,N-1,$ can be parametrized in terms of $[C_N]$. Since all the rate constants and dissociation constants are positive, it is easy to see that there exists a set of positive concentrations $\{[T],[M], [C_0], \ldots, [C_N]\}$ for which the right hand sides of the equations (\ref{eq:McK}) vanish. This implies that McKeithan's network is complex-balanced. On the other hand, all reactions in this network are {\it irreversible}, and thus McKeithan's network is not detailed-balanced.

The main aim of this paper is to show how the compact mathematical formulation of detailed-balanced chemical reaction networks derived in \cite{Ours} can be extended to complex-balanced networks (such as McKeithan's network). Indeed, the crucial difference between detailed-balanced and complex-balanced networks will turn out to be that in the latter case the Laplacian matrix is {\it not symmetric} anymore, but still {\it balanced} (in the sense of the terminology used in graph theory and multi-agent dynamics). In particular, complex-balanced chemical networks will be shown to bear close resemblance with asymmetric consensus dynamics with balanced Laplacian matrix. Exploiting this formulation it will be shown how the dynamics of complex-balanced networks share important common characteristics with those of detailed-balanced networks, including a similar characterization of the set of all equilibria and the same stability result stating that the system converges to an equilibrium point uniquely determined by the initial condition. For the particular case when the complex-stoichiometric matrix $Z$ is injective as in the McKeithan's network case (see Remark \ref{remark}), the same asymptotic stability results were obtained before in \cite{Sontag}. %\footnote{Similar results for the specific example of McKeithan's network were already derived in \cite{Sontag}.}. 
Furthermore, while for detailed-balanced networks it has been shown in \cite{Ours} that all equilibria are in fact thermodynamic equilibria in this paper the similar result will be proved that all equilibria of a complex-balanced network are complex-equilibria. Similar results have already been proved in \cite{HornJackson}; however, the proofs presented in the current paper are much more concise and insightful as compared to those presented in \cite{HornJackson}.

Furthermore, based on our formulation of complex-balanced networks exhibiting a balanced weighted Laplacian matrix associated to the graph of complexes, we will propose a technique for model-reduction of complex-balanced networks. This technique is similar to the Kron reduction method for model reduction of resistive electrical networks described in \cite{Kron}; see also \cite{Doerfler, vdsSCL}. Our technique works by deleting complexes from the graph of complexes associated with the network. In other words, our reduced network has fewer complexes and usually fewer reactions as compared to the original network, and yet the behavior of a number of significant metabolites in the reduced network is approximately the same as in the original network. Thus our model reduction method is useful from a computational point of view, specially when we need to deal with models of large-scale chemical reaction networks. Mathematically our approach is based on the result that the Schur complement (with respect to the deleted complexes) of the balanced weighted Laplacian matrix of the full graph of complexes is again a balanced weighted Laplacian matrix corresponding to the reduced graph of complexes.

%In a related study of graph-theoretic application in chemical networks by Angeli {\it et al.} in \cite{Angeli2010}, a bipartite graph, the so-called {\it species-reactant} graph, is defined where two types of nodes are considered, namely, the species and the reaction nodes, and the edges describe the relation between these two nodes. Using the structural property of the graph, Angeli {\it et al.} obtain a qualitative convergence result of the network exploiting the monotonicity of the reaction rates.

The paper is organized as follows. In Section 2, we introduce tools from stoichiometry of reactions and graph theory that are required to derive our mathematical formulation. In Section 3, we explain mass action kinetics, define complex-equilibria and complex-balanced networks and then derive our formulation. In Section 4, we derive equilibrium and stability properties of complex-balanced networks using our formulation. In Section 5, we propose a model reduction method for complex-balanced networks, while Section 6 presents conclusions based on our results.

\bigskip

\noindent\emph{\bf Notation}:  The space of ${n}$ dimensional real vectors is denoted by $\mathbb{R}^{{n}}$,
and the space of ${m}\times {n}$ real matrices by $\mathbb{R}^{{m}\times {n}}$. The space of ${n}$ dimensional real vectors consisting of all strictly positive entries is denoted by $\mR_+^{n}$ and the space of ${n}$ dimensional real vectors consisting of all nonnegative entries is denoted by $\bar{\mR}_+^{n}$. The rank of a real matrix $A$ is denoted by $\rank A$. dim$(\mathcal{V})$ denotes the dimension of a set $\mathcal{V}$. Given $a_1,\ldots,a_n \in \mR$, $\mbox{diag}(a_1,\ldots,a_n)$ denotes the diagonal matrix with diagonal entries $a_1,\ldots,a_n$; this notation is extended to the block-diagonal case when $a_1,\ldots,a_n$ are real square matrices. Furthermore, $\ker A$ and $\spa A$ denote the kernel and span respectively of a real matrix $A$. If $U$ denotes a linear subspace of $\mathbb{R}^m$, then $U^{\perp}$ denotes its orthogonal subspace (with respect to the standard Euclidian inner product). $\mathds{1}_m$ denotes a vector of dimension $m$ with all entries equal to 1. The time-derivative $\frac{dx}{dt}(t)$ of a vector $x$ depending on time $t$ will be usually denoted by $\dot{x}$. 

Define the mapping
$\mathrm{Ln} : \mathbb{R}_+^m \to \mathbb{R}^m, \quad x \mapsto \mathrm{Ln}(x),$
as the mapping whose $i$-th component is given as
$\left(\mathrm{Ln}(x)\right)_i := \mathrm{ln}(x_i).$
Similarly, define the mapping
$\mathrm{Exp} : \mathbb{R}^m \to \mathbb{R}_+^m, \quad x \mapsto \mathrm{Exp}(x),$
as the mapping whose $i$-th component is given as
$\left(\mathrm{Exp}(x)\right)_i := \mathrm{exp}(x_i).$
Also, define for any vectors $x,z \in \mathbb{R}^m$ the vector $x \cdot z \in \mathbb{R}^m$ as the element-wise product $\left(x \cdot z\right)_i :=x_iz_i, \, i=1,2,\ldots,m,$ and the vector $\frac{x}{z} \in \mathbb{R}^m$ as the element-wise quotient $\left(\frac{x}{z}\right)_i := \frac{x_i}{z_i}, \, i=1,\cdots,m$. Note that with these notations $\Exp (x + z) = \Exp (x) \cdot \Exp (z)$ and $\Ln (x \cdot z) = \Ln (x) + \Ln (z), \Ln \left(\frac{x}{z}\right) = \Ln (x) - \Ln (z)$.

\section{Chemical reaction network structure}
In this section, we introduce the tools necessary in order to derive our mathematical formulation of the dynamics of complex-balanced networks. First we introduce the concept of stoichiometric matrix of a reaction network. We then define the concept of a graph of complexes, which was first introduced in the work of Horn \& Jackson and Feinberg (\cite{HornJackson, Horn, Feinberg, Feinberg1}). 

\subsection{Stoichiometry}\label{sec:stoich}
Consider a chemical reaction network involving $m$ chemical species (metabolites), among which $r$ chemical reactions take place. The basic structure underlying the dynamics of the vector $x \in \mathbb{R}_+^m$ of concentrations $x_i, i=1,\cdots, m,$ of the chemical species is given by the {\it balance laws} $\dot{x} = Sv$, where $S$ is an $m \times r$ matrix, called the {\it stoichiometric matrix}. The elements of the vector $v \in \mathbb{R}^r$ are commonly called the (reaction) {\it fluxes} or {\it rates}. The stoichiometric matrix $S$, which consists of (positive and negative) integer elements, captures the basic conservation laws of the reactions. It already contains useful information about the network dynamics, {\it independent} of the precise form of the reaction rate $v(x)$. Note that the reaction rate depends on the governing law prescribing the dynamics of the reaction network.

We now show how to construct the stoichiometric matrix for a reaction network with the help of an example shown in Figure \ref{fig:eg1}.
\begin{figure}[h]
\centerline{
  \scalebox{1}{
     \input{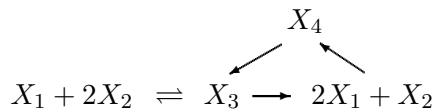}
     }
   }
\caption{Example of a reaction network}
\label{fig:eg1}
\end{figure}

Note that the above reaction has 3 irreversible and 1 reversible reaction leading to in total 5 reactions among four species $X_1$, $X_2$, $X_3$ and $X_4$. Since the stoichiometric matrix maps the space of reactions to the space of species, it has dimension $4 \times 5$. The entry of $S$ corresponding to the $\ith$ row and $\jth$ column is obtained by subtracting the number of moles of the $\ith$ species on the product side from that on the substrate side for the $\jth$ reaction. Thus
\[
S=\begin{bmatrix}
-1 & 1 & 2 & -2 & 0 \\
-2 & 2 & 1 & -1 & 0 \\
1 & -1 & -1 & 0 & -1 \\
0 & 0 & 0 & 1 & 1
\end{bmatrix}
\]
for the reaction network depicted in Figure \ref{fig:eg1}.

%In many cases of interest, especially in biochemical reaction networks, chemical reaction networks are intrinsically {\it open}, in the sense that there is a continuous exchange with the environment (in particular, inflow and outflow of chemical species and connection to other reaction networks). This will be modeled by a vector $v_b \in \mathbb{R}^b$ consisting of $b$ {\it boundary} (or, exchange) {\it fluxes}, leading to an extended model
%\begin{equation}\label{stoichiometry}
%\dot{x} = Sv+ S_b v_b
%\end{equation}
%Here the matrix $S_b$ consists of mutually different columns whose elements are all $0$ except for one element equal to $1$ or $-1$. Thus the boundary fluxes $v_b$ are uptake ($1$) or demand ($-1$) reactions for part of the chemical species, called the {\it boundary chemical species}. 

In this paper, we focus only on closed chemical reaction networks meaning those without external fluxes. Therefore unless otherwise mentioned, our chemical reaction networks do not have any external fluxes.

If there exists an $m$-dimensional row-vector $k$ such that $kS = 0$, then the quantity $kx$ is a {\it conserved quantity} or a {\it conserved moiety} for the dynamics $\dot{x} = Sv(x)$ for {\it all} possible reaction rates $v=v(x)$. Indeed, $\frac{d}{dt} kx = kSv(x) =0$. Later on, in Remark \ref{rem:LCM}, it will be shown that law of conservation of mass leads to a conserved moiety of a chemical reaction network.

Note that for all possible fluxes the solutions of the differential equations $\dot{x} = Sv(x)$ starting from an initial state $x_0$ will always remain within the affine space
\begin{equation}\label{eq:PSCC} 
\mathcal{S}_{x_0}:=\{x \in \mR_+^m \mid x-x_0 \in \im S\}.
\end{equation}
$\mathcal{S}_{x_0}$ has been referred to as the \emph{positive stoichiometric compatibility class} (corresponding to $x_0$) in \cite{Feinberg1, Siegel, GAC}. 

\subsection{Graph of complexes}
The structure of a chemical reaction network cannot be directly captured by an ordinary graph. Instead, we will follow an approach originating in the work of Horn and Jackson \cite{HornJackson}, introducing the space of {\it complexes}. The set of complexes of a chemical reaction network is simply defined as the union of all the different left- and righthand sides (substrates and products) of the reactions in the network. Thus, the complexes corresponding to the network (\ref{fig:eg1}) are $X_1+2X_2$, $X_3$, $2X_1+X_2$ and $X_4$.

The expression of the complexes in terms of the chemical species is formalized by an $m \times c$ matrix $Z$, whose $\alpha$-th column captures the expression of the $\alpha$-th complex in the $m$ chemical species. For example, for the network depicted in Figure \ref{fig:eg1},
\[
Z=\begin{bmatrix} 1 & 0 & 2 & 0 \\ 2 & 0 & 1 & 0 \\0 & 1 & 0 & 0 \\ 0 & 0 & 0 & 1\end{bmatrix}.
\]
We will call $Z$ the {\it complex-stoichiometric matrix} of the network. Note that by definition all elements of the matrix $Z$ are non-negative integers.

Since the complexes are the left- and righthand sides of the reactions, they can be naturally associated with the vertices of a {\it directed graph} $\mathcal{G}$ with edges corresponding to the reactions. Formally, the reaction $ \alpha  \longrightarrow \beta$ between the $\alpha$-th (reactant) and the $\beta$-th (product) complex defines a directed edge with tail vertex being the $\alpha$-th complex and head vertex being the $\beta$-th complex. The resulting graph will be called the {\it graph of complexes}.

Recall, see e.g. \cite{Bollobas}, that any graph is defined by its {\it incidence matrix} $B$. This is a $c \times r$ matrix, $c$ being the
number of vertices and $r$ being the number of edges, with $(\alpha,j)$-th
element equal to $-1$ if vertex $\alpha$ is the tail vertex of edge $j$ and $1$ if vertex $\alpha$ is the head vertex of edge $j$, while $0$ otherwise. %Define furthermore the \emph{origin matrix} $Y$ of the graph of complexes as a $c \times r$ matrix with $(\alpha,j)$-th
%element equal to $1$ if vertex $\alpha$ is the tail vertex of edge $j$, while $0$ otherwise. 

Obviously, there is a close relation between the matrix $Z$ and the stoichiometric matrix $S$. In fact, it is easily checked that
\begin{equation}\label{eq:vel}
S = ZB, \quad \text{hence} \quad \dot{x}=ZBv(x)
\end{equation} 
with $B$ the incidence matrix of the graph of complexes.

\section{The dynamics of complex-balanced networks governed by mass action kinetics}
In this section, we first recall the dynamics of species concentrations of reactions governed by mass action kinetics. We then define complex-balanced networks and derive a compact mathematical formulation for their dynamics.

\subsection{The general form of mass action kinetics}
Recall that for a chemical reaction network, the relation between the reaction rates and species concentrations depends on the governing laws of the reactions involved in the network. In this section, we explain this relation for reaction networks governed by mass action kinetics.
In other words, if $v$ denotes the vector of reaction rates and $x$ denotes the species concentration vector, we show how to construct $v(x)$. The reaction rate of the $j$-th reaction of a mass action chemical reaction network, from a substrate complex $\mathcal{S}_j$ to a product complex $\mathcal{P}_j$, is given as
\begin{equation}\label{eq:massaction}
v_j(x) = k_j\prod_{i=1}^m x_i ^{Z_{i \mathcal{S}_j}},
\end{equation}
where $Z_{i \rho}$ is the $(i,\rho)$-th element of the complex-stoichiometric matrix $Z$, and $k_j$ $ \geq 0$ is the rate constant of the $j$-th reaction. Without loss of generality we will assume throughout that for every $j$, the constant $k_j$ is positive. 

This can be rewritten in the following way. Let $Z_{\mathcal{S}_j}$ denote the column of the complex-stoichiometric matrix $Z$ corresponding to the substrate complex of the $j$-th reaction. Using the mapping $\mathrm{Ln}  : \mathbb{R}^c_+ \to \mathbb{R}^c$ as defined at the end of the Introduction, the mass action reaction equation (\ref{eq:massaction}) for the $j$-th reaction from substrate complex $\mathcal{S}_j$ to product complex $\mathcal{P}_j$ can be rewritten as
\begin{equation}\label{eq:massactionwithZ}
v_j(x)=k_{j} \exp\big(Z_{\mathcal{S}_j}^T \mathrm{Ln}(x)\big).
\end{equation}
Based on the formulation in (\ref{eq:massactionwithZ}), we can describe the complete reaction network dynamics as follows. Let the mass action rate for the complete set of reactions be given by the vector $v(x)= \begin{bmatrix} v_1(x) & \cdots & v_r(x) \end{bmatrix}^T$. 
For every $\sigma,\pi \in \{1,\ldots, c\}$, define
\[
\mathcal{C}_{\pi \sigma}:=\left\{j \in \{1,\ldots,r\} \mid (\sigma,\pi)=(\mathcal{S}_j,\mathcal{P}_j)\right\}
\]
and $a_{\pi \sigma}:=\sum_{j\in \mathcal{C}_{\pi \sigma}}k_{j}$. Thus if there is no reaction $\sigma\rightarrow \pi$, then $a_{\pi \sigma}=0$. Define the {\it weighted adjacency matrix} $A$ of the graph of complexes as the matrix with $(\pi,\sigma)$-th element $a_{\pi \sigma}$, where $\pi,\sigma \in \{1, \ldots,c\}$.
Furthermore, define $L := \Delta - A$, where $\Delta$ is the diagonal matrix whose $(\rho,\rho)$-th element is equal to the sum of the elements of the $\rho$-th column of $A$. Let $B$ denote the incidence matrix of the graph of complexes associated with the network. By definition of $L$, we have
$\mathds{1}_c^T L = 0$. It can be verified that the vector $Bv(x)$ for the mass action reaction rate vector $v(x)$ is equal to $-L \Exp \left(Z^T \Ln(x) \right)$, where the mapping $\Exp : \mathbb{R}^c \to \mathbb{R}^c_+$ has been defined at the end of the Introduction.
Hence the dynamics can be compactly written as
\begin{equation}\label{sontag}
\dot{x} = - Z L \mathrm{Exp} \left(Z^T \mathrm{Ln}(x)\right)
\end{equation}
A similar expression of the dynamics corresponding to mass action kinetics, in less explicit form, was already obtained in \cite{Sontag}.

\subsection{Complex-balanced networks}
We now define a class of reaction networks known as complex-balanced networks. This class was first defined in the work of Horn and Jackson (see p. 92 of \cite{HornJackson}). We first define a complex-equilibrium of a reaction network.
\begin{definition}
Consider a chemical reaction network with dynamics given by the equation $($\ref{eq:vel}$)$. A vector of concentrations $x^* \in \mR_+^{m}$ is called a \emph{complex-equilibrium} if $Bv(x^*)=0$. Furthermore, a chemical reaction network is called \emph{complex-balanced} if it admits a complex-equilibrium.
\end{definition}
It is easy to see that any complex-equilibrium is an equilibrium for the network, but the other way round need not be true (since $Z$ need not be injective). We now explain the physical interpretation of a complex-equilibrium. Observe that 
\begin{equation}\label{eq:CE}
Bv(x^*)=0
\end{equation}
consists of $c$ equations where $c$ denotes the number of complexes. Among all the reactions that the $\ith$ complex $C_i$ is involved in, let $\mathcal{O}_i$ denote the set of all the reactions for which $C_i$ is the substrate complex and let $\mathcal{I}_i$ denote the set of all reactions for which $C_i$ is the product complex.  The $\ith$ of equations (\ref{eq:CE}) can now be written as 
\[
\sum_{k \in \mathcal{I}_i}v_k(x^*)=\sum_{k \in \mathcal{O}_i}v_k(x^*)
\]
It follows that at a complex-equilibrium, the combined rate of outgoing reactions from any complex is equal to the combined rate of incoming reactions to the complex. In other words, at a complex-equilibrium, every complex involved in the network is at equilibrium.

%From equation (\ref{eq:ith}), it follows that every complex of a complex-balanced network is a substrate for at least one reaction of the network. 

In \cite{Horn} and \cite{Dick}, conditions have been derived for a chemical reaction network governed by mass action kinetics to be complex-balanced.

\begin{remark}\rm
A \emph{thermodynamically balanced} or \emph{detailed-balanced} chemical reaction network is one for which there exists a vector of positive species concentrations $x^*$ at which each of the reactions of the network is at equilibrium, that is, $v(x^*)=0$, see e.g. \cite{Ours}. Such networks are necessarily reversible. Clearly every thermodynamically balanced network is complex-balanced. 
\end{remark}

We now rewrite the dynamical equations for complex-balanced networks governed by mass action kinetics in terms of a known complex-equilibrium. It will be shown that such a form of equations has advantages in deriving stability properties of and also a model reduction method for complex-balanced networks. Recall equation (\ref{sontag}) for general mass action reaction networks:
\begin{equation}\label{eq:Sontag}
\dot{x} = -Z L \mathrm{Exp} \left(Z^T \mathrm{Ln}(x)\right)
\end{equation}
Assume that the network is complex-balanced with a complex-equilibrium $x^*$. 
Define 
\[
K(x^*):= \text{diag}_{i=1}^{c}\big(\exp(Z_i^T \Ln (x^*))\big)
\]
where $Z_i$ denotes the $\ith$ column of $Z$. Equation (\ref{eq:Sontag}) can be rewritten as
\begin{equation}\label{eq:standform}
\dot{x} = -ZLK(x^*)K(x^*)^{-1}\Exp\big(Z^T\Ln (x)\big)=-Z\mL(x^*)\Exp\left(Z^T\Ln\left(\frac{x}{x^*}\right)\right)
\end{equation}
where $\mL(x^*):=LK(x^*)$. Note that since $\mathds{1}_c^TL=0$, also $\mathds{1}_c^T\mL(x^*)=0$. Furthermore since $x^*$ is a complex-equilibrium, we have
%HERE I HAVE CHANGED THE ORDER
\[
\mL(x^*)\mathds{1}_c = \mathcal{L}(x^*)\Exp\left(Z^T\Ln\left(\frac{x}{x^*}\right)\right)_{\mid_{x=x^*}}=0
\]
Hereafter, we refer to $\mL(x^*)$ as the \emph{weighted Laplacian} of the graph of complexes associated with the given complex-balanced network. Both the row and column sums of the weighted Laplacian $\mL(x^*)$ are equal to zero \footnote{In the literature on directed graphs (see e.g. \cite{CDC}), $\mL$ is called \emph{balanced}. Note that the matrix $L$, having zero column sums but {\it not} zero row sums, is similar to the `advection' set-up considered in \cite{CDC}.}. It is this special property of the weighted Laplacian that we make use of in deriving all the results stated further on in this paper.

\subsection{The linkage classes of a graph of complexes} \label{sec:link}
A \emph{linkage class} of a chemical reaction network is a maximal set of complexes $\{\mathcal{C}_1,\ldots,\mathcal{C}_k\}$ such that $\mathcal{C}_i$ is connected by reactions to $\mathcal{C}_j$ for every $i,j \in \{1,\ldots,k\}, i \neq j$. It can be easily verified that the number of linkage classes $(\ell)$ of a network, which is equal to the number of connected components of the graph of complexes corresponding to the network, is given by $\ell =c-$ rank$(B)$ (the number of linkage classes in the terminology of \cite{HornJackson, Feinberg, Feinberg1}). The graph of complexes is connected, i.e., there is one linkage class in the network if and only if ker$(B^T)=\text{span}\big(\mathds{1}_c\big)$.

Assume that the reaction network has $\ell$ linkage classes. Assume that the $\ith$ linkage class has $r_i$ reactions between $c_i$ complexes. Partition $Z$, $B$ and $\mL$ matrices according to the various linkage classes present in the network as follows:
\begin{eqnarray*}
Z &=& \begin{bmatrix}
Z_1 & Z_2 & \ldots & Z_{\ell}
\end{bmatrix}\\
B &=& \begin{bmatrix}
B_1 & 0 & 0 & \ldots & 0\\
0 & B_2 & 0 & \ldots & 0\\
\vdots & \vdots & \ddots & \vdots & \vdots \\
0 & \ldots & 0 & B_{\ell -1} & 0\\
0 & \ldots & \ldots & 0 & B_{\ell}
\end{bmatrix}\\
\mL(x^*) &=& \text{diag}(\mL_1(x^*),\mL_2(x^*), \ldots, \mL_{\ell -1}(x^*),\mL_{\ell}(x^*))
\end{eqnarray*}
where for $(i=1,\ldots,\ell)$, $Z_i \in \bar{\mR}_+^{m \times c_i}$, $B_i \in \mR^{c_i \times r_i}$ and $\mL_i(x^*) \in \mR^{c_i \times c_i}$ denote the complex-stoichiometric matrix, incidence matrix and the weighted Laplacian matrices corresponding to equilibrium concentration $x^*$ respectively for the $\ith$ linkage class. Let $S_i$ denote the stoichiometric matrix of the $\ith$ linkage class. It is easy to see that $S_i=Z_iB_i$. Observe that equation (\ref{eq:standform}) can be written as
\begin{equation}\label{eq:link}
\dot{x}=-\sum_{i=1}^{\ell}Z_i\mL_i(x^*)\Exp\left(Z_i^T\Ln\left(\frac{x}{x^*}\right)\right)
\end{equation}

\begin{remark}\rm \label{rem:LCM}
The \emph{law of conservation of mass} states that there exists $u \in \mR_+^m$, such that $Z_i^{T}u \in \spa \big(\mathds{1}_{c_i}\big)$ for $i=1,\ldots,\ell$. This implies that $u^{T}x$ is a conserved moiety for the dynamics $\dot{x} = ZBv$, for {\it all} forms of the reaction rate $v(x)$. Indeed, $Z_i^{T}u \in \spa \big(\mathds{1}_{c_i}\big)$ implies $u^{T}ZB = 0$, since $B_i^T \mathds{1}_{c_i} =0$.
\end{remark}

\section{Equilibria and stability of complex-balanced networks}
In this section, we make use of the compact mathematical formulation (\ref{eq:standform}) in order to derive properties of equilibria and stability of complex-balanced networks. 
\subsection{Equilibria}
Our first result is a characterization of the set of all equilibria of a complex-balanced network in terms of a known equilibrium.
\begin{theorem}\label{th:char}
Consider a complex-balanced network governed by mass action kinetics. Let $S \in \mR^{m \times r}$ denote the stoichiometric matrix and assume that $x^* \in \mR_+^{m}$ is a complex-equilibrium for the network. The following hold:
\begin{enumerate}
\item $x^{**}\in \mR_+^{m}$ is another equilibrium for the network iff $S^{T}\Ln\left(\frac{x^{**}}{x^*}\right)=0$.
\item Every equilibrium of the network is a complex-equilibrium.
\end{enumerate}
\end{theorem}
The proof of this theorem crucially makes use of the following lemma, which will also be the basis for the proof of Theorem \ref{th:Lyap}.
\begin{lemma}\label{lemma}
Let $\mL(x^*)$ be a balanced weighted Laplacian matrix as before. Then for any $\gamma \in \mathbb{R}^c$
\begin{equation*}%\label{eq:gamma}
\gamma^T\mL(x^*)\Exp(\gamma) \geq 0
\end{equation*}
Moreover
\begin{equation*}%\label{eq:gamma1}
\gamma^T\mL(x^*)\Exp(\gamma) = 0 \mbox{ if and only if } B^T \gamma =0
\end{equation*}
\end{lemma}
\begin{proof}
Let $\gamma_i$ denote the $\ith$ element of $\gamma$ and let $k_{ij}$ denote the negative of the element of $\mL(x^*)$ corresponding to the $\jth$ row and $\ith$ column. Note that $k_{ij} \geq 0, i,j=1,\cdots,c$. Using $\mathds{1}_c^T\mL(x^*)=0$ the expression $-\gamma^T\mL(x^*)\Exp(\gamma)$ can be rewritten as
\begin{eqnarray*}
\begin{bmatrix}
\sum_{i\neq 1}(\gamma_i-\gamma_1)k_{1i} & \sum_{i\neq 2}(\gamma_i-\gamma_2)k_{2i} & \ldots & \sum_{i\neq c}(\gamma_i-\gamma_c)k_{ci}
\end{bmatrix}\Exp(\gamma) \\
 = \sum_{j=1}^{c}\sum_{i\neq j}(\gamma_i-\gamma_j)k_{ji}\text{exp}(\gamma_j)
\end{eqnarray*}
Furthermore, since the exponential function is strictly convex
\begin{equation}
(\beta - \alpha) \text{exp}(\alpha) \leq \text{exp}(\alpha) - \text{exp}(\beta)
\end{equation}
for all $\alpha, \beta$, with equality if and only if $\alpha=\beta$. Hence 
\begin{eqnarray}\label{ineq:1}
-\gamma^T\mL(x^*)\Exp(\gamma) = \sum_{j=1}^{c}\sum_{i\neq j}(\gamma_i-\gamma_j)k_{ji}\text{exp}(\gamma_j) & \leq & \sum_{j=1}^{c} \sum_{i\neq j}k_{ji}\big(\text{exp}(\gamma_i)-\text{exp}(\gamma_j)\big)\\
\nonumber & = & \mathds{1}_c^T \begin{bmatrix}
\sum_{i \neq 1} k_{1i}\big(\text{exp}(\gamma_i)-\text{exp}(\gamma_1)\big) \\
\vdots \\
\sum_{i \neq c} k_{ci}\big(\text{exp}(\gamma_i)-\text{exp}(\gamma_c)\big) \\
\end{bmatrix}\\
\nonumber & = & -\mathds{1}_c^T\mL(x^*)^T\Exp(\gamma) \\
\nonumber & = & -\big(\mL(x^*)\mathds{1}_c\big)^T\Exp(\gamma)=0.
\end{eqnarray}
since $\mL(x^*)$ is balanced. 

Furthermore, equality occurs in inequality (\ref{ineq:1}) only when each of the terms within the summation on the left hand side is equal to the corresponding term within the summation on the right hand side. Since $k_{ij}>0, i \neq j,$ if the $\ith$ complex reacts to the $\jth$ complex, 
%SHODHAN: YOU HAD 'or vice versa' here; is that correct ?
it follows that $\gamma_i=\gamma_j$ for each such $i,j$, which is equivalent to $B^T\gamma=0$. 
\end{proof}

\begin{proof} (of Theorem \ref{th:char})
The dynamics of the complex-balanced network with $c$ complexes are given by
\[
\dot{x}=-Z\mL(x^*)\Exp\left(Z^T\Ln\left(\frac{x}{x^*}\right)\right)
\]
where $Z \in \bar{\mR}^{m \times c}_+$ and $\mL(x^*) \in \mR^{c \times c}$ are as defined in the previous section. Let $B \in \mR^{c \times r}$ denote the incidence matrix of the graph of complexes associated with the network. Assume that the reaction network has $\ell$ linkage classes. Assume that the $\ith$ linkage class has $r_i$ reactions between $c_i$ complexes. Partition $Z$, $B$ and $\mL$ matrices according to the various linkage classes present in the network as in Section \ref{sec:link}. Define $S_i:=Z_iB_i$ for $i=1,\ldots,\ell$.

\medskip{}

(\emph{1}.)

(\emph{Only If}): Assume that $x^{**}$ is an equilibrium, that is
\begin{equation}\label{eq:dx}
%\frac{dx}{dt}\mid_{x=x^{**}}=
-Z\mL(x^*)\Exp \left(Z^T\Ln \left(\frac{x^{**}}{x^*}\right) \right)=0
\end{equation} 
Define $\gamma:=Z^T\Ln\left(\frac{x^{**}}{x^*}\right)$. Premultiplying equation (\ref{eq:dx}) with $\Ln\left(\frac{x^{**}}{x^*}\right)$, we get
\begin{equation*}%\label{eq:gamma}
-\gamma^T\mL(x^*)\Exp(\gamma)=0.
\end{equation*}
Hence by Lemma \ref{lemma}, $B^T\gamma=0$ and thus $S^T\Ln\left(\frac{x^{**}}{x^*}\right)=0$.
\medskip{}

(\emph{If}) Assume that $S^T\Ln\left(\frac{x^{**}}{x^*}\right)=0$. Hence for every linkage class $i=1,\ldots,\ell$, $S_i^T\Ln\left(\frac{x^{**}}{x^*}\right)=0$, or, equivalently, $B_i^T\gamma = 0$. This implies that $\gamma_i=\gamma_j$ if the $\ith$ complex reacts to the $\jth$ complex or vice-versa. This in turn implies that for every linkage class $i=1,\ldots,\ell$, $Z_i^T\Ln\left(\frac{x^{**}}{x^*}\right)$ consists of equal entries, or in other words it can be written as 
\[
Z_i^T\Ln\left(\frac{x^{**}}{x^*}\right)=\Gamma_i \mathds{1}_{c_i}
\] 
where $\Gamma_i \in \mR$.

Since $x^*$ is a complex-equilibrium, $\mL(x^*)\mathds{1}_{c}=0$. This implies that for $i=1,\ldots,\ell$,  $\mL_i(x^*)\mathds{1}_{c_i}=0$. Now, by evaluating the RHS of (\ref{eq:link}) at $x^{**}$, we have 
\[
%\frac{dx}{dt}\mid_{x=x^{**}}
-\sum_{i=1}^{\ell}Z_i\mL_i(x^*)\Exp\left(Z_i^T\Ln\left(\frac{x^{**}}{x^*}\right)\right)=-\sum_{i=1}^{\ell}\text{exp}(\Gamma_i)Z_i\mL_i(x^*)\mathds{1}_{c_i}=0. 
\]
\medskip{}

(\emph{2}.) Let $x^{**}$ denote an equilibrium as in the proof of the earlier part. Then $S^T\Ln\left(\frac{x^{**}}{x^*}\right)=0$. We prove that $x^{**}$ is a complex-equilibrium. As shown earlier, $Z_i^T\Ln\left(\frac{x^{**}}{x^*}\right)=\Gamma_i \mathds{1}_{c_i}$ for $i=1,\ldots,\ell$. Since $\mL_i(x^*)\mathds{1}_{c_i}=0$, this implies that (c.f., the discussion that preceeds (\ref{sontag}) and the form in (\ref{eq:link}))
\[
-Bv(x^{**})=\mL(x^*)\Exp\left(Z^T\Ln\left(\frac{x^{**}}{x^*}\right)\right)=\begin{bmatrix}
\mL_1(x^*)\Exp\left(Z_1^T\Ln\left(\frac{x^{**}}{x^*}\right)\right)\\
\vdots \\
\mL_{\ell}(x^*)\Exp\left(Z_{\ell}^T\Ln\left(\frac{x^{**}}{x^*}\right)\right)
\end{bmatrix}=0
\]
From the above equation, it follows that $x^{**}$ is a complex-equilibrium.
\end{proof}

\begin{remark}\rm \label{rem:ZD}
The steps followed in the proof of the above theorem are very similar to the proof of the characterization of the space of equilibria of a class of networks known as \emph{zero-deficiency networks} presented in \cite{Feinberg1}. In the next subsection, we define zero-deficiency networks and prove that every zero-deficiency network that admits an equilibrium is complex-balanced. We emphasize here that the proof of Theorem \ref{th:char} that is presented in this paper is much more simple as compared to similar proofs provided in \cite{Feinberg1} due to the use of the properties of the weighted Laplacian in the present manuscript.
\end{remark}

%ADDED !
One may wonder to what extent the balanced weighted Laplacian matrix $\mL(x^*)$ depends on the choice of the complex-equilibrium $x^*$. This dependency turns out to be very minor, strengthening the importance of this matrix for the analysis of the network. Indeed, consider any other complex-equilibrium $x^{**}$. Then $S^T \Ln x^{**} = S^T \Ln x^*$, or equivalently $B^TZ^T \Ln x^{**} = B^TZ^T \Ln x^*$. Hence for the $i$-th connected component of the complex graph we have $B_i^TZ_i^T \Ln x^{**} = B_i^TZ_i^T \Ln x^*$, or equivalently, since $\ker B_i^T = \spa \mathds{1}$,
\begin{equation}
Z_i^T \Ln x^{**} = Z_i^T \Ln x^* + c_i \mathds{1}
\end{equation}
for some constant $c_i$. Thus from the definition of $\mL$, it follows that $\mL_i(x^{**}) = d_i \mL_i(x^*)$ for some positive constant $d_i$. Hence, on every connected component of the graph of complexes, the balanced weighted Laplacian matrix $\mL(x^*)$ is {\it unique} up to multiplication by a positive constant.  
\subsection{Zero-deficiency networks}
We now introduce the notion of zero-deficient chemical reaction networks mentioned in Remark \ref{rem:ZD}. This notion was introduced in the work of Feinberg \cite{Feinberg2} in order to relate the stoichiometry of a given network to the structure of the associated graph of complexes. 
\begin{definition}
The deficiency $\delta$ of a chemical reaction network with complex-stoichiometric matrix $Z$, incidence matrix $B$ and stoichiometric matrix $S$ is defined as
\begin{equation}\label{eq:def}
\delta := \rank (B) - \rank (ZB) = \rank (B) - \rank (S) \geq 0
\end{equation}
A reaction network has {\it zero-deficiency} if $\delta =0$. 
\end{definition}
Note that zero-deficiency is equivalent with ker$(Z)\cap$ im$(B)=0$, or with the mapping
\[
Z: \im B \subset \mathbb{R}^c \to \mathbb{R}^m
\] 
being {\it injective}. 
\begin{remark}\rm
The deficiency of a chemical reaction network has been defined in a different way in \cite{Feinberg2}. Denote by $\ell$ the number of linkage classes of a given chemical reaction network. Note that $\ell=c-\text{rank}(B)$ as explained in Section \ref{sec:link}. In \cite{Feinberg2}, deficiency $\delta$ is defined as
\begin{equation}\label{eq:def1}
\delta:=c-\ell-\text{rank}(S)
\end{equation}
It is easy to see that definitions (\ref{eq:def}) and (\ref{eq:def1}) are equivalent.
\end{remark}

We now prove that every zero-deficiency network that admits an equilibrium is complex-balanced. Consequently all the results that we state for a complex-balanced network also hold for a zero-deficiency network that admits an equilibrium.
\begin{lemma}
If a chemical reaction network is zero-deficient and admits an equilibrium, then it is complex-balanced.
\end{lemma}
\begin{proof}
Consider a zero-deficient network with complex-stoichiometric matrix $Z \in \bar{\mR}_+^{m \times c}$ and incidence matrix $B \in \mR^{c \times r}$. Let $x^* \in \mR_+^{m}$ denote an equilibrium for the given network. Then $Sv(x^*)=ZBv(x^*)=0$ and hence by zero-deficiency $Bv(x^*)=0$. Consequently $x^*$ is a complex-equilibrium and the network is complex-balanced.
\end{proof}

The above lemma has been stated and proved earlier in \cite[Theorem 4.1, p. 192]{Feinberg2} in a different and more lengthy manner.
\begin{remark}\rm \label{remark}
For McKeithan's network it is easily seen that $Z$ itself is already injective, thus implying zero-deficiency.
\end{remark}

\subsection{Asymptotic stability}
We now show global asymptotic stability of complex-balanced networks.
\begin{theorem}\label{th:Lyap}
Consider a complex-balanced network with stoichiometric matrix $S \in \mR^{m \times r}$, an equilibrium $x^* \in \mR_+^m$ and dynamics given by equation $($\ref{eq:standform}$)$. Assume that the network obeys the law of conservation of mass stated in Remark \ref{rem:LCM}. Then for every initial concentration vector $x_0 \in \mR_+^m$, the species concentration vector $x$ converges as $t \rightarrow \infty$ to an element of the set
\begin{equation}\label{eq:epsilon}
\mathcal{E} :=\{x^{**} \in \mR_+^m \mid S^T\Ln (x^{**}) = S^T\Ln(x^*)\}.
\end{equation}
\end{theorem}
\begin{proof}
Define
\begin{equation}\label{eq:Lyap}
G(x)=x^T\Ln\left(\frac{x}{x^{*}}\right)+(x^*-x)^T\mathds{1}_m
\end{equation}
Observe that $G(x^*)=0$\footnote{$G$ defined by (\ref{eq:Lyap}) is a standard Lyapunov function used in chemical reaction network theory (see for example \cite{Feinberg1}).}. We prove that 
\begin{equation}\label{eq:lyap}
G(x) > 0  \quad \forall x \neq x^{*},
\end{equation}
and is {\it proper}, i.e., for every real $ c >0$ the set $\{x\in \bar{\mR}_+^{m} \mid G(x) \leq c\}$ is compact. Let $x_i$ and $x_i^*$ denote the $\ith$ elements of $x$ and $x^*$ respectively. From the strict concavity of the logarithmic function, 
\begin{equation}\label{eq:z}
z-\ln(z) \geq 1
\end{equation}
$\forall$ $z \in \mR_+$ with equality occuring only when $z=1$. Putting $z=\frac{x_i^*}{x_i}$ in equation (\ref{eq:z}), we get
\[
x_i^*-x_i+x_i\ln\left(\frac{x_i}{x_i^*}\right) \geq 0
\] 
This implies that
\[
G(x)=\sum_{i=1}^m\left[x_i^*-x_i+x_i\ln\left(\frac{x_i}{x_i^*}\right)\right] \geq 0.
\]
with equality occuring only when $x=x^*$, thus proving (\ref{eq:lyap}). Properness of $G$ is readily checked.

We first prove that 
\begin{equation}\label{lyap1}
\dot{G}(x) := \frac{\partial G}{\partial x}(x)\dot{x} \leq 0 \qquad \forall x \in \mR_+^{m},
\end{equation}
and
\begin{equation}\label{lyap2}
\dot{G}(x) = 0 \mbox{ if and only if } x\in \mathcal{E}.
\end{equation}
%Let $c$ denote the number of complexes in the network and for $i=1,\ldots,c$, let $C_i$ denote the $\ith$ complex. 
Observe that
\[
\dot{G}(x) %= \frac{dG(x)}{dt} = \left(\frac{dG(x)}{dx}\right)^T\frac{dx}{dt}
=-\Ln\left(\frac{x}{x^*}\right)^TZ^T\mL(x^*)\Exp\left(Z^T\Ln\left(\frac{x}{x^*}\right)\right)
\]
Defining $\gamma:=Z^T\Ln\left(\frac{x}{x^*}\right)$ we thus obtain
\[
\dot{G}(x)=-\gamma^T\mL(x^*)\Exp(\gamma)
\]
and the statement follows from Lemma \ref{lemma}.
%As in the proof of the (\emph{Only If}) part of Theorem \ref{th:char}, let $\gamma_i$ denote the $\ith$ element of $\gamma$ and let $k_{ij}$ denote the negative of the element of $\mL(x^*)$ corresponding to the $\jth$ row and $\ith$ column. Since $\mathds{1}_c^T\mL(x^*)=0$, it follows that 
%\begin{eqnarray*}
%-\gamma^T\mL(x^*)\Exp(\gamma) 
%& = & \begin{bmatrix}
%\sum_{i\neq 1}(\gamma_i-\gamma_1)k_{1i} & \sum_{i\neq 2}(\gamma_i-\gamma_2)k_{2i} & \ldots & \sum_{i\neq c}(\gamma_i-\gamma_c)k_{ci}
%\end{bmatrix}\Exp(\gamma) \\
%& = & \sum_{j=1}^{c}\sum_{i\neq j}(\gamma_i-\gamma_j)k_{ji}\text{exp}(\gamma_j)
%\end{eqnarray*}
%Since the exponential function is convex,
%\begin{eqnarray*}
%\dot{G}(x)=\sum_{j=1}^{c}\sum_{i\neq j}(\gamma_i-\gamma_j)k_{ji}\text{exp}(\gamma_j) & \leq & \sum_{j=1}^{c}\sum_{i\neq j}k_{ji}\big(\text{exp}(\gamma_i)-\text{exp}(\gamma_j)\big)\\
%& = & \mathds{1}_c^T \begin{bmatrix}
%\sum_{i \neq 1} k_{1i}\big(\text{exp}(\gamma_i)-\text{exp}(\gamma_1)\big) \\
%\vdots \\
%\sum_{i \neq c} k_{ci}\big(\text{exp}(\gamma_i)-\text{exp}(\gamma_c)\big) \\
%\end{bmatrix}\\
%& = & -\big(\mL(x^*)\mathds{1}_c\big)^T\Exp(\gamma)=0
%\end{eqnarray*}
%with equality occuring only when $\gamma_i=\gamma_j$ for every pair of reacting complexes $C_i,C_j$. This occurs when $B^T\gamma=B^TZ^T\Ln\left(\frac{x}{x^*}\right)=S^T\Ln\left(\frac{x}{x^*}\right)=0$, equivalently when $x \in \mathcal{E}$. This proves (\ref{lyap1}) and (\ref{lyap2}).
%

We now prove that $\mR_+^m$ is forward invariant with respect to (\ref{eq:standform}). Assume by contradiction that this is not the case, and that $x_{i}(t)=0$ for some $t$ and $i \in \{1, \ldots, m\}$. Without loss of generality assume that the species with concentration $x_i$ is present in at least one complex which is involved in a reaction. Let $\mathcal{C}_i$ be the subset of complexes which contain $x_i$,  and denote by $Z_{\mathcal{C}_i}$ the matrix containing the columns of $Z$ corresponding to the complexes in $\mathcal{C}_i$. Hence all elements of the $i$-th row $Z_{\mathcal{C}_i}$ are different from zero. Then it follows that $\prod_{j=1}^m x_j^{Z_{jC}}=0$ if $C \in \mathcal{C}_i$ and thus 
\begin{equation*}
\dot x_{i}(t)  = - Z^{i} \mL(x^*) \begin{bmatrix}
\prod_{j=1}^m x_j(t)^{Z_{jC_1}}\\
\vdots \\
\prod_{j=1}^m x_j(t)^{Z_{jC_c}}
\end{bmatrix} > 0,
\end{equation*}  
where $Z^i$ is the $i$-th row vector of $Z$. This inequality is due to the fact that the terms corresponding to the positive $i$-th diagonal element of the weighted Laplacian matrix $\mL(x^*)$ are all zero, while there is at least one term corresponding to a non-zero, and therefore strictly negative, off-diagonal element of $\mL(x^*)$. It is easy to see that the inequality also holds at points arbitrarily close to the boundaries of the positive orthant $\mR_+^m$ except the origin. Recall that according to the law of conservation of mass, there exists $u \in \mR_+^m$ such that $u^Tx$ is a conserved moiety. Consequently, starting from an initial concentration vector $x_0\in \mR_+^m$, the state trajectory $x(\cdot)$ can not reach the origin. This implies that $\mR_+^m$ must be forward invariant with respect to (\ref{eq:standform}). 

Since $G$ is proper (in $\bar{\mR}_+^{m}$) and the state trajectory $x(\cdot)$ remains in $\mR_+^m$, (\ref{lyap1}) implies that $x(\cdot)$ is bounded in $\mR_+^m$. Therefore, boundedness of $x(\cdot)$, together with equations (\ref{lyap1}) and (\ref{lyap2}), imply that the species concentration $x$ converges to an element of the set $\mathcal{E}$ by an application of LaSalle's invariance principle.
\end{proof}

\begin{remark}\rm
The crux of the proofs of Theorems \ref{th:char} and \ref{th:Lyap} is the inequality $\gamma^T \mL(x^*) \Exp (\gamma) \geq 0, \text{ for all } \gamma \in \mR$.
%\begin{equation}\label{eq:crux} 
%\end{equation} 
This inequality holds because of balancedness of $\mL$ and we make use of the convexity of the exponential function in order to prove it. 
\end{remark}

\begin{remark}\rm
By proving that $\mR_+^m$ is forward invariant with respect to (\ref{eq:standform}), we have actually proved the \emph{persistence conjecture} stated in \cite[p. 1488]{GAC} of complex-balanced networks obeying the law of conservation of mass. Persistence conjecture roughly corresponds to the requirement of nonextinction of species concentration of a complex-balanced network. Forward invariance of $\mR_+^m$ with respect to (\ref{sontag}) has been proved in \cite[Section VII]{Sontag} for the case when the complex-stoichiometric matrix $Z$ is injective. Proving invariance of $\mathbb{R}^m_+$ is an intricate problem for general mass action kinetics; see e.g. \cite{Angeli2011} and the references quoted therein.
\end{remark}

\subsection{Equilibrium concentration corresponding to an initial concentration}
In this section, we show that corresponding to every positive stoichiometric compatibility class (see equation (\ref{eq:PSCC}) for a definition) of a complex-balanced chemical reaction network, there exists a unique complex-equilibrium in $\mathcal{E}$ defined by equation (\ref{eq:epsilon}). The proof that we provide for this result is very similar to the proof of the zero-deficiency theorem provided in \cite{Feinberg1} and is based on the following proposition therein. Recall from the Introduction that the product $x \cdot z \in \mathbb{R}^m$ is defined element-wise. 
\begin{proposition}\label{prop:uniq}
Let $U$ be a linear subspace of $\mR^m$, and let $x^*,x_0 \in \mR_+^m$. Then there is a unique element $\mu \in U^{\perp}$, such that $\big(x^* \cdot \Exp(\mu)-x_0\big) \in U$.
\end{proposition}
\begin{proof}
See proof of \cite[Proposition B.1, pp. 361-363]{Feinberg1}.
\end{proof} 
\begin{theorem}\label{th:unique}
Consider the complex-balanced chemical reaction network with dynamics given by equation $($\ref{eq:standform}$)$ and equilibrium set $\mathcal{E}$. Then for every $x_0 \in \mR_+^m$ there exists a unique $x_1 \in \mathcal{E} \cap \mathcal{S}_{x_0}$ with $\mathcal{S}_{x_0}$ given by $($\ref{eq:PSCC}$)$, and the solution trajectory $x(\cdot)$ of  $($\ref{eq:standform}$)$ with initial condition $x(0)=x_0$ converges for $t \to \infty$ to $x_1$.
\end{theorem}
\begin{proof}
With reference to Proposition \ref{prop:uniq}, define $U=\im S$, and observe that $U^{\perp}=\ker S^T$. By Proposition \ref{prop:uniq}, there exists a unique $\mu \in \ker S^T$ such that $x^* \cdot \Exp(\mu)-x_0 \in \im S$. Define $x_1:=x^* \cdot \Exp(\mu) \in \mR_+^m$. It follows that $S^{T}\mu=S^T\Ln\left(\frac{x_1}{x^*}\right)=0$, i.e., $x_1 \in \mathcal{E}$. Also, since $x_1-x_0 \in \im S$, $x_1 \in \mathcal{S}_{x_0}$ which is an invariant set of the dynamics.
Together with Theorem \ref{th:Lyap} it follows that the state trajectory $x(\cdot)$ with initial condition $x(0)=x_0$ converges to the equilibrium $x_1 \in \mathcal{E}$. 
\end{proof}

\section{Model reduction}

For biochemical reaction networks, model-order reduction is still underdeveloped. The singular perturbation method and quasi steady-state approximation (QSSA) approach are the most commonly used techniques, where the reduced state contains a part of the metabolites of the full model. In the thesis by H\"ardin \cite{Hardin2010}, the QSSA approach is extended by considering higher-order approximation in the computation of quasi steady-state. Sunn\aa ker {\it et al.} in \cite{Sunnaker2011} proposed a reduction method by identifying variables that can be lumped together and can be used to infer back the original state. In Prescott \& Papachristodoulou \cite{Prescott2012}, a method to compute the upper-bound of the error is proposed based on sum-of-squares programming. The application of these techniques to general kinetics laws, such as Michaelis-Menten, poses a significant computational problem.  

%For many purposes one may wish to reduce the number of dynamical equations of a chemical reaction network in such a way that the behaviour of a number of key metabolites is approximated in a satisfactory way. 

In this section, we propose a novel and simple method for model reduction of complex-balanced chemical reaction networks governed by mass action kinetics. Our method is based on the Kron reduction method for model reduction of resistive electrical networks described in \cite{Kron}; see also \cite{vdsSCL}. Moreover, the resulting reduced-order model retains the structure of the kinetics and gives result to a reduced biochemical reaction network, which enables a direct biochemical interpretation. 

\subsection{Description of the method}
Our model reduction method is based on \emph{reduction of the graph of complexes} associated with the network. It is based on the following result regarding Schur complements of weighted Laplacian matrices.
\begin{proposition}\label{prop:WL}
Consider a complex-balanced network with a complex-equilibrium $x^* \in \mR_+^m$ and weighted Laplacian matrix $\mL(x^*) \in \mR^{c \times c}$ corresponding to the equilibrium $x^*$. Let $\mathcal{V}$ denote the set of vertices of the graph of complexes associated with the network. Then for any subset of vertices $\mathcal{V}_r \subset \mathcal{V}$, the Schur complement $\hat{\mL}(x^*)$ of $\mL(x^*)$ with respect to the indices corresponding to $\mathcal{V}_r$ satisfies the following properties:
\begin{enumerate}
\item All diagonal elements of $\hat{\mL}(x^*)$ are positive and off-diagonal elements are nonnegative.
\item $\mathds{1}_{\hat{c}}^T\hat{\mL}(x^*)=0$ and $\hat{\mL}(x^*)\mathds{1}_{\hat{c}}=0$, where $\hat{c}:=c-\text{dim}(\mathcal{V}_r)$.
\end{enumerate} 
\end{proposition}
\begin{proof}
(\emph{1}.) Follows from the proof of \cite[Theorem 3.11]{Niezink}; see also \cite{vdsSCL} for the case of symmetric $\mL$.

\medskip{}

(\emph{2}.) Without loss of generality, assume that the last $c-\hat{c}$ rows and columns of $\mL(x^*)$ correspond to the vertex set $\mathcal{V}_r$. Consider a partition of $\mL(x^*)$ given by
\begin{equation}\label{eq:part}
\mL(x^*)=\begin{bmatrix}
\mL_{11}(x^*) & \mL_{12}(x^*)\\
\mL_{21}(x^*) & \mL_{22}(x^*)
\end{bmatrix}
\end{equation}
where $\mL_{11}(x^*) \in \mR^{\hat{c}\times \hat{c}}$, $\mL_{12}(x^*) \in \mR^{\hat{c} \times (c-\hat{c})}$, $\mL_{21}(x^*)\in \mR^{(c-\hat{c})\times \hat{c}}$ and $\mL_{22}(x^*)\in \mR^{(c-\hat{c})\times(c-\hat{c})}$. By definition,
\[
\hat{\mL}(x^*)=\mL_{11}(x^*)-\mL_{12}(x^*)\mL_{22}(x^*)^{-1}\mL_{21}(x^*)
\]
Since $\mathds{1}_{c}^T\mL(x^*)=0$, we obtain
\begin{eqnarray*}
\mathds{1}_{\hat{c}}^T\mL_{11}(x^*)+\mathds{1}_{c-\hat{c}}^T\mL_{21}(x^*) &=& 0\\
\mathds{1}_{\hat{c}}^T\mL_{12}(x^*)+\mathds{1}_{c-\hat{c}}^T\mL_{22}(x^*) &=& 0
\end{eqnarray*}
Using the above equations, we get
\begin{eqnarray*}
\mathds{1}_{\hat{c}}^T\hat{\mL}(x^*) &=& \mathds{1}_{\hat{c}}^T\big(\mL_{11}(x^*)-\mL_{12}(x^*)\mL_{22}(x^*)^{-1}\mL_{21}(x^*)\big) \\ &=& -\mathds{1}_{c-\hat{c}}^T\mL_{21}(x^*)+\mathds{1}_{c-\hat{c}}^T\mL_{22}(x^*)\mL_{22}(x^*)^{-1}\mL_{21}(x^*)=0
\end{eqnarray*}
In a similar way, it can be proved that $\hat{\mL}(x^*)\mathds{1}_{\hat{c}}=0$.
\end{proof}

From the above result, it follows that $\hat{\mL}(x^*)$ obeys all the properties of the weighted Laplacian of a complex-balanced chemical network corresponding to the graph of complexes with vertex set $\mathcal{V}-\mathcal{V}_r$. Thus Proposition \ref{prop:WL} can be directly applied to the graph of complexes, yielding a reduction of the chemical reaction network by reducing the number of complexes. Consider a complex-balanced reaction network described in the standard form (\ref{eq:standform})
\[
\Sigma: \quad \dot{x} = - Z \mL(x^*) \mathrm{Exp} \left(Z^T \mathrm{Ln}\left(\frac{x}{x^*}\right)\right)
\]
Reduction will be performed by {\it deleting certain complexes in the graph of complexes}, resulting in a reduced graph of complexes with weighted Laplacian $\hat{\mL}(x^*)$. Furthermore, leaving out the corresponding columns of the complex-stoichiometric matrix $Z$ one obtains a reduced complex-stoichiometric matrix $\hat{Z}$ (with as many columns as the remaining number of complexes in the graph of complexes), leading to the reduced reaction network
\begin{equation}\label{reduced}
\hat{\Sigma}: \quad \dot{x} = - \hat{Z} \hat{\mL}(x^*) \Exp \left(\hat{Z}^T \Ln \left(\frac{x}{x^*}\right)\right).
\end{equation}
Note that $\hat{\Sigma}$ is again a {\it complex-balanced chemical reaction network} governed by mass action kinetics, with a reduced number of complexes and with, in general, a different set of reactions (edges of the graph of complexes). Furthermore, the complex-equilibrium $x^*$ of the original reaction network $\Sigma$ is a complex-equilibrium of the reduced network $\hat{\Sigma}$ as well. 

An interpretation of the reduced network $\hat{\Sigma}$ can be given as follows. Consider a subset $\mathcal{V}_r$ of the set of all complexes, which we wish to leave out in the reduced network. Consider the partition of $\mL(x^*)$ as given by equation (\ref{eq:part}) and a corresponding partition of $Z$ given by
\begin{equation}\label{partition}
Z = \begin{bmatrix} Z_1 & Z_2 \end{bmatrix},
\end{equation}
where $\mathcal{V}_r$ corresponds to the last part of the indices (denoted by $2$), in order to write out the dynamics of $\Sigma$ as
\[
\dot{x} = - \begin{bmatrix} Z_1 & Z_2 \end{bmatrix}  \begin{bmatrix} \mL_{11}(x^*) & \mL_{12}(x^*) \\ \mL_{21}(x^*) & \mL_{22}(x^*) \end{bmatrix} 
\begin{bmatrix} \Exp \left( Z_1^T \Ln \left(\frac{x}{x^*}\right)\right) \\ \Exp \left(Z_2^T \Ln \left(\frac{x}{x^*}\right)\right) \end{bmatrix}
\]
Consider now the auxiliary dynamical system
\[
\begin{bmatrix} \dot{y}_1 \\ \dot{y}_2 \end{bmatrix} = - \begin{bmatrix} \mL_{11}(x^*) & \mL_{12}(x^*) \\ \mL_{21}(x^*) & \mL_{22}(x^*) \end{bmatrix} 
\begin{bmatrix} w_1 \\ w_2 \end{bmatrix}
\]
where we impose the constraint $\dot{y}_2 =0$. It follows that 
\[
w_2 = - \mL_{22}(x^*)^{-1}\mL_{21}(x^*)w_1, 
\]
leading to the reduced auxiliary dynamics defined by the Schur complement
\[
\dot{y}_1 = - \big( \mL_{11}(x^*) - \mL_{12}(x^*)\mL_{22}(x^*)^{-1}\mL_{21}(x^*) \big) w_1 = - \hat{\mL}(x^*) w_1
\]
Putting back in $w_1 = \Exp \left(\hat{Z}_1^T \Ln \left(\frac{x}{x^*}\right)\right)$, making use of 
$
\dot{x} = Z_1 \dot{y}_1 + Z_2 \dot{y}_2 = Z_1 \dot{y}_1 = \hat{Z} \dot{y}_1
$, we then obtain the reduced network $\hat{\Sigma}$ given in (\ref{reduced}).

\smallskip 

We derive the following properties relating $\Sigma$ and $\hat{\Sigma}$.

\begin{proposition}\label{prop:reduction}
Consider the complex-balanced reaction network $\Sigma$ and its reduced order model $\hat{\Sigma}$ given by $($\ref{reduced}$)$. Denote their sets of equilibria by $\mathcal{E}$, respectively $\hat{\mathcal{E}}$. Then $\mathcal{E} \subseteq \hat{\mathcal{E}}$. %Furthermore, if $\Sigma$ has deficiency zero then so does $\hat{\Sigma}$.
\end{proposition}
\begin{proof}
Let $B$ denote the incidence matrix for $\Sigma$ and let $c$ and $\hat{c}$ denote the number of complexes in $\Sigma$ and $\hat{\Sigma}$ respectively. Assume that the graph of complexes is connected; otherwise the same argument can be repeated for every component (linkage class). It follows that ker$(B^T)=$ span$(\mathds{1}_c)$. Let $x^{**} \in \mathcal{E}$. We show that $x^{**} \in \hat{\mathcal{E}}$. Let $Z$ and $\hat{Z}$ denote the complex-stoichiometric matrices of $\Sigma$ and $\hat{\Sigma}$ respectively. Let $\mL(x^*)$ denote the weighted Laplacian of the graph of complexes of $\Sigma$ corresponding to an equilibrium $x^*$. Let $\hat{\mL}(x^*)$ denote the Schur complement of $\mL(x^*)$ corresponding to the reduced model $\hat{\Sigma}$.

Since $x^{**} \in \mathcal{E}$, $B^TZ^T\Ln\left(\frac{x^{**}}{x^*}\right)=0$. It follows that $Z^T\Ln\left(\frac{x^{**}}{x^*}\right)\in $ span$(\mathds{1}_c)$. This implies that $\hat{Z}^T\Ln\left(\frac{x^{**}}{x^*}\right)\in $ span$(\mathds{1}_{\hat{c}})$ since the columns of $\hat{Z}$ form a subset of the columns of $Z$. From Proposition \ref{prop:WL}, it now follows that $\hat{\mL}(x^*)\Exp\left(\hat{Z}^T\Ln\left(\frac{x^{**}}{x^*}\right)\right)=0$. Consequently $x^{**}\in \hat{\mathcal{E}}$. This concludes the proof.

\end{proof}

\subsection{Effect of Model Reduction}\label{sec:effect}
In this section, we show the effect of our model reduction method on two types of complex-balanced networks with a single linkage class. In other words, we give an interpretation of our reduced model in terms of its corresponding full model for two types of networks. Note that deletion of a set of complexes in one linkage class does not affect the reactions of the other linkage classes of the network. 

\medskip{}

{\bf Type 1:} 
\begin{equation}\label{Type1}
\text{Full Network:} \qquad \mathcal{C}_1 \overset{k_1}{\underset{k_{-1}}{\rightleftharpoons}} \mathcal{C}_2 \overset{k_2}{\underset{k_{-2}}{\rightleftharpoons}} \mathcal{C}_3 \overset{k_3}{\underset{k_{-3}}{\rightleftharpoons}} \cdots \cdots \overset{k_{n-1}}{\underset{k_{-(n-1)}}{\rightleftharpoons}} \mathcal{C}_n
\end{equation}
\begin{equation}\label{Type1red}
\text{Reduced Network:} \qquad \mathcal{C}_1 \overset{\frac{k_1k_2}{k_{-1}+k_2}}{\underset{\frac{k_{-1}k_{-2}}{k_{-1}+k_2}}{\rightleftharpoons}} \mathcal{C}_3 \overset{k_3}{\underset{k_{-3}}{\rightleftharpoons}} \cdots \cdots \overset{k_{n-1}}{\underset{k_{-(n-1)}}{\rightleftharpoons}} \mathcal{C}_n
\end{equation}

These are complex-balanced networks with reversible reactions occuring between consecutive elements of the set of distinct complexes $\{\mathcal{C}_1,\mathcal{C}_2, \ldots, \mathcal{C}_n\}$ as in (\ref{Type1}). The reduced network obtained by deleting the complex $\mathcal{C}_2$ can be computed to be (\ref{Type1red}). The two reactions, $\mathcal{C}_1 \rightleftharpoons \mathcal{C}_2$ and $\mathcal{C}_2 \rightleftharpoons \mathcal{C}_3$ in the full network are replaced by one reaction $\mathcal{C}_1 \rightleftharpoons \mathcal{C}_3$ in the reduced network. This reaction is again a reversible reaction governed by mass action kinetics, with rate constants given by (\ref{Type1red}). 

     The transient behaviour of the metabolites involved in the complexes of the reduced model will approximately be the same as that of the full model if the metabolites involved in $\mathcal{C}_2$ reach their steady states much faster than the remaining metabolites. In this case, we can safely impose the condition that the metabolites involved in $\mathcal{C}_2$ are at constant concentration in order to obtain the reduced model (\ref{Type1red}) with similar transient behaviour as that of (\ref{Type1}). The rule of induction may be applied in order to further reduce the model by deleting more complexes. 

A special case of Type 1 networks is $\mathcal{C}_1 \rightleftharpoons \mathcal{C}_2$. Deletion of the complex $\mathcal{C}_2$ in this case is equivalent to deletion of the linkage class from the network. Such a deletion provides a close approximation to the original network if the reaction has very little effect on the dynamics of the network, i.e., if the reaction reaches its steady state much faster than the remaining reactions of the network.  

\medskip{}

\begin{figure}[ht]
\begin{minipage}[b]{0.45\linewidth}
\centerline{
  \scalebox{0.6}{
     \ifx\JPicScale\undefined\def\JPicScale{1}\fi
\unitlength \JPicScale mm
\begin{picture}(111,24)(0,0)
\put(0,0){\makebox(0,0)[cc]{$\mathcal{C}_0$}}

\linethickness{0.3mm}
\put(6,0){\line(1,0){11}}
\put(17,0){\vector(1,0){0.12}}
\put(20,0){\makebox(0,0)[cc]{$\mathcal{C}_1$}}

\linethickness{0.3mm}
\put(24,0){\line(1,0){12}}
\put(36,0){\vector(1,0){0.12}}
\put(40,0){\makebox(0,0)[cc]{$\mathcal{C}_2$}}

\linethickness{0.3mm}
\put(47,0){\line(1,0){12}}
\put(59,0){\vector(1,0){0.12}}
\put(64,0){\makebox(0,0)[cc]{$\mathcal{C}_3$}}

\linethickness{0.3mm}
\put(70,0){\line(1,0){12}}
\put(82,0){\vector(1,0){0.12}}
\put(87,0){\makebox(0,0)[cc]{$\ldots$}}

\linethickness{0.3mm}
\put(93,0){\line(1,0){12}}
\put(105,0){\vector(1,0){0.12}}
\put(111,0){\makebox(0,0)[cc]{$\mathcal{C}_n$}}

\linethickness{0.3mm}
\qbezier(19,2)(14.32,5.14)(11.07,5.98)
\qbezier(11.07,5.98)(7.82,6.82)(5.5,5.5)
\qbezier(5.5,5.5)(3.15,4.2)(2.07,3.59)
\qbezier(2.07,3.59)(0.99,2.99)(1,3)
\put(9,-3){\makebox(0,0)[cc]{$k_{1}$}}

\put(10,3){\makebox(0,0)[cc]{$k_{-1}$}}

\linethickness{0.3mm}
\multiput(1,3)(0.12,0.25){8}{\line(0,1){0.25}}
\linethickness{0.3mm}
\multiput(1,3)(0.25,-0.12){8}{\line(1,0){0.25}}
\linethickness{0.3mm}
\qbezier(39,2)(29.14,7.24)(20.24,7.72)
\qbezier(20.24,7.72)(11.33,8.21)(2,4)
\linethickness{0.3mm}
\qbezier(63,4)(47.44,10.3)(32.76,10.3)
\qbezier(32.76,10.3)(18.08,10.3)(2,4)
\linethickness{0.3mm}
\qbezier(110,3)(74.11,14.02)(48.12,14.26)
\qbezier(48.12,14.26)(22.14,14.5)(2,4)
\put(26,4){\makebox(0,0)[cc]{$k_{-2}$}}

\put(44,7){\makebox(0,0)[cc]{$k_{-3}$}}

\put(73,10){\makebox(0,0)[cc]{$k_{-n}$}}

\put(31,-3){\makebox(0,0)[cc]{$k_{2}$}}

\put(52,-3){\makebox(0,0)[cc]{$k_{3}$}}

\put(76,-3){\makebox(0,0)[cc]{$k_{4}$}}

\put(98,-3){\makebox(0,0)[cc]{$k_{n}$}}

\end{picture}
     }
   }
\caption{Type 2 full network}
\label{fig:Type3}
\end{minipage}
\hspace{0.5cm}
\begin{minipage}[b]{0.45\linewidth}
\centerline{
  \scalebox{0.6}{
     \ifx\JPicScale\undefined\def\JPicScale{1}\fi
\unitlength \JPicScale mm
\begin{picture}(104,46)(0,0)
\put(0,0){\makebox(0,0)[cc]{$\mathcal{C}_0$}}

\linethickness{0.3mm}
\put(6,0){\line(1,0){25}}
\put(31,0){\vector(1,0){0.12}}
\put(37,0){\makebox(0,0)[cc]{$\mathcal{C}_1$}}

\linethickness{0.3mm}
\put(40,0){\line(1,0){10}}
\put(50,0){\vector(1,0){0.12}}
\put(54,0){\makebox(0,0)[cc]{$\mathcal{C}_3$}}

\linethickness{0.3mm}
\put(59,0){\line(1,0){12}}
\put(71,0){\vector(1,0){0.12}}
\put(79,0){\makebox(0,0)[cc]{$\ldots$}}

\linethickness{0.3mm}
\put(86,0){\line(1,0){12}}
\put(98,0){\vector(1,0){0.12}}
\put(104,0){\makebox(0,0)[cc]{$\mathcal{C}_n$}}

\linethickness{0.3mm}
\qbezier(37,2)(27.63,10.37)(21.14,12.42)
\qbezier(21.14,12.42)(14.64,14.46)(10,10.5)
\qbezier(10,10.5)(5.3,6.59)(3.14,4.78)
\qbezier(3.14,4.78)(0.97,2.98)(1,3)
\put(16,-3){\makebox(0,0)[cc]{$k_{1}$}}

\put(27,14){\makebox(0,0)[cc]{$k_{-3}$}}

\linethickness{0.3mm}
\multiput(1,3)(0.12,0.25){8}{\line(0,1){0.25}}
\linethickness{0.3mm}
\multiput(1,3)(0.25,-0.12){8}{\line(1,0){0.25}}
\linethickness{0.3mm}
\qbezier(53,3)(39.51,17.17)(27.24,17.41)
\qbezier(27.24,17.41)(14.97,17.65)(2,4)
\linethickness{0.3mm}
\qbezier(103,3)(71.29,25.57)(46.99,25.81)
\qbezier(46.99,25.81)(22.68,26.05)(2,4)
\put(18,6){\makebox(0,0)[cc]{$k_{-1}+\frac{k_2k_{-2}}{k_{-2}+k_3}$}}

\put(68,18){\makebox(0,0)[cc]{$k_{-n}$}}

\put(45,-4){\makebox(0,0)[cc]{$\frac{k_2k_3}{k_{-2}+k_3}$}}

\put(65,-3){\makebox(0,0)[cc]{$k_{4}$}}

\put(92,-3){\makebox(0,0)[cc]{$k_{n}$}}

\end{picture}
     }
   }
\caption{Type 2 reduced network}
\label{fig:Type3red}
\end{minipage}
\end{figure}
{\bf Type 2:} These are complex-balanced networks between distinct complexes $\{\mathcal{C}_1,\mathcal{C}_2, \ldots, \mathcal{C}_n\}$ as shown in Figure \ref{fig:Type3}. McKeithan's network is an example of such networks. The reduced network obtained by deleting complex $\mathcal{C}_2$ is shown in Figure \ref{fig:Type3red}. Observe that the reaction $\mathcal{C}_1 \rightarrow \mathcal{C}_0$ in the reduced network has a different rate constant as compared to the full network. The three reactions $\mathcal{C}_1 \rightarrow \mathcal{C}_2$, $\mathcal{C}_2 \rightarrow \mathcal{C}_3$ and $\mathcal{C}_2 \rightarrow \mathcal{C}_0$ of the full network are replaced by one reaction $\mathcal{C}_1 \rightarrow \mathcal{C}_3$ in the reduced network. The rate constant for this reaction is given in Figure \ref{fig:Type3red}. All the remaining reactions of the reduced network occur in the same way as in the full network.

In this case, the transient behaviour of the metabolites involved in the complexes of the reduced model will approximately be the same as that of the full model if the metabolites involved in $\mathcal{C}_2$ reach their steady states much faster than the remaining metabolites. Using the method described in this paper, we can study the effects of deleting other complexes like $\mathcal{C}_1$ or $\mathcal{C}_n$ from the model.
\medskip{}
 
Observe that for all the types of networks discussed above, it is important to determine which of the complexes are to be deleted so that the reduced model approximates the full model reasonably well.

\begin{example}\rm
We have applied the model reduction method described in this paper to the model of T-cell interactions as in (\ref{eq:McK}). We use the following numerical values: $N=19$; 
\[
\begin{array}{ccccc}
k_{p,0}=52; & k_{p,1}=49; & k_{p,2}=41; & k_{p,3}=39; & k_{p,4}=37; \\
k_{p,5}=34; & k_{p,6}=31; & k_{p,7}=29; & k_{p,8}=25; & k_{p,9}=19; \\
k_{p,10}=16; & k_{p,11}=21; & k_{p,12}=20; & k_{p,13}=19; & k_{p,14}=18; \\
k_{p,15}=15; & k_{p,16}=24; & k_{p,17}=13; & k_{p,18}=7; & k_{p,19}=5; \\
k_{-1,0}=13; & k_{-1,1}=29; & k_{-1,2}=0.16; & k_{-1,3}=1.4; & k_{-1,4}=2.3; \\
k_{-1,5}=2; & k_{-1,6}=0.19; & k_{-1,7}=0.33; & k_{-1,8}=0.94; & k_{-1,9}=0.67; \\
k_{-1,10}=0.31; & k_{-1,11}=0.21; & k_{-1,12}=3; & k_{-1,13}=5; & k_{-1,14}=1; \\
k_{-1,15}=11; & k_{-1,16}=0.8; & k_{-1,17}=7; & k_{-1,18}=1; & k_{-1,19}=17.
\end{array}
\]
The initial value of each of the complexes $C_i$, $i=0,\ldots,19$ is assumed to be equal to 0.01. The complexes $T$ and $M$ are assumed to have initial concentrations $1$ and $2$ respectively. We have performed model reduction by deleting 5 complexes $C_{15}, C_{16}, C_{17}, C_{18}$ and $C_{19}$. We have simulated the transient behaviour of the remaining complexes for the first two time units. The simulation results show that there is a good agreement between the transient behaviour of the concentration of most of such complexes when comparing the full network to the reduced network. Figure \ref{fig:graph1} depicts plots of comparison of the concentration profiles of $T$ and $M$. 
\begin{figure}[t]
\begin{center}
\scalebox{0.7}{
\includegraphics[width=\columnwidth]{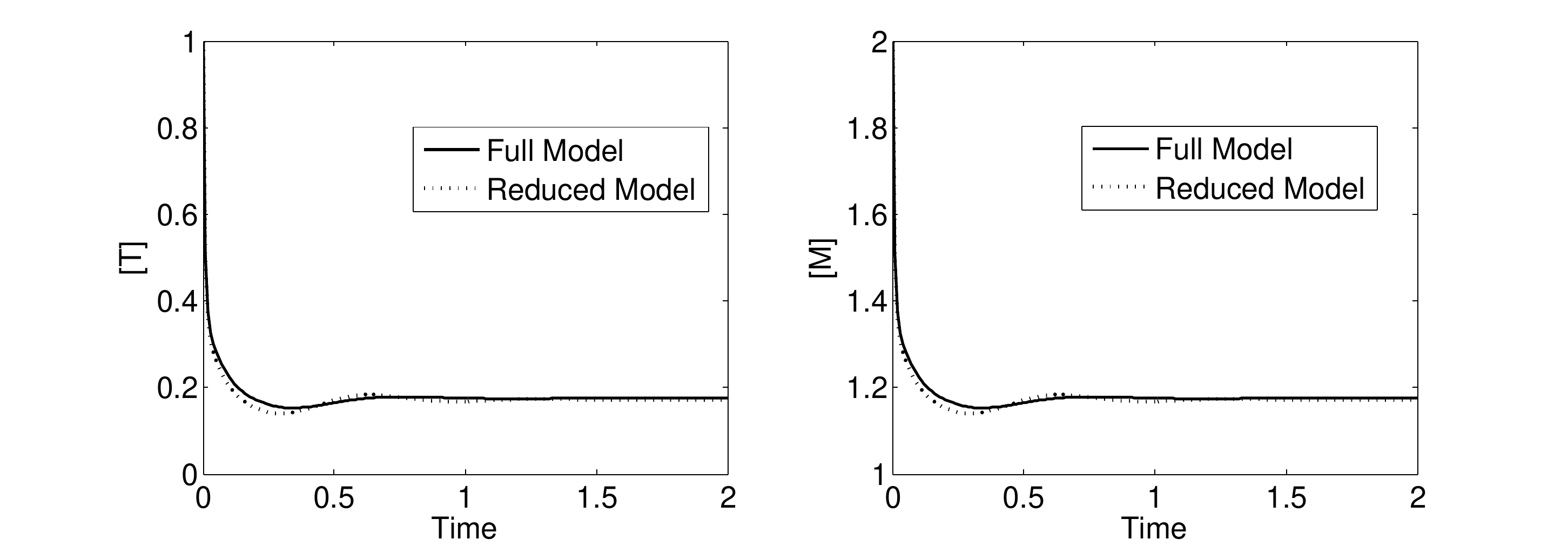}}
\end{center}
\caption{Concentration profiles of $T$ and $M$}
\label{fig:graph1}
\end{figure}

An interpretation of model reduction of the above example model is as follows. By deleting complexes $C_{15}, C_{16}, C_{17}, C_{18}$ and $C_{19}$, we assume that these complexes are at constant concentration. Since deleting these complexes results in a reduced model that closely mimics the original model, it follows that in the full model, these complexes reach an equilibrium much faster than the remaining complexes, so that assuming that these complexes are at constant concentrations results in a close approximation of the original model.
\end{example}

\section{Conclusion}
In this paper, we have provided a compact mathematical formulation for the dynamics of complex-balanced networks. We have made use of this formulation for the determination of equilibria and the asymptotic stability of such networks. The methods that have been employed are very similar to the ones used in \cite{Feinberg1}, but the difference is that our proofs are much more concise than the ones presented in \cite{Feinberg1} due to the use of properties of balanced weighted Laplacian matrices of complex-balanced networks. By proving the forward invariance of the positive orthant with respect to the dynamics of complex-balanced networks, we believe that we have proved the persistence conjecture for complex-balanced networks stated in \cite{GAC}. Furthermore, we have made use of the formulation in order to derive a model reduction technique for complex-balanced networks. 

A main challenge for further research is the extension of our results to chemical reaction networks with external fluxes and/or externally controlled concentrations. This will change the stability analysis considerably, due to the nonlinearity of the differential equations. Furthermore, it will also lead to scrutinizing the model reduction technique proposed in this paper from an external (input-output) point of view.

%We have applied the model reduction technique proposed in this paper in order to reduce the yeast glycolysis model described in \cite{Karen}. This model is a reversible biochemical reaction network and can be proved to be a complex-balanced network. Although not governed by mass action kinetics, it can be shown that our model reduction method is applicable for this network. We have simulated the transient behaviour of the species that were not eliminated during the model reduction procedure. It was found that there is a good agreement between the transient behaviour of the concentration of most of such species when comparing the full network to the reduced network.


\begin{thebibliography}{99.}
\bibitem{GAC}
D.F. Anderson, ``A proof of the global attractor conjecture in the single linkage class case", {\it SIAM J. Appl. Math.}, 71(4), pp. 1487--1508, 2011.

\bibitem{Angeli2010}
D. Angeli, P. De Leenheer, E.D. Sontag, ``Graph-theoretic characterizations of monotonicity of chemical networks in reaction coordinates", J. Math. Biol., 61, pp. 581--616, 2010.

\bibitem{AngeliEJC}
D. Angeli, ``A tutorial on chemical reaction network dynamics", {\it European Journal of Control}, 15 (3-4), pp. 398--406, 2009.

\bibitem{Angeli2011}
D. Angeli, P. De Leenheer, E.D. Sontag,
``Persistence results for chemical reaction networks with time-dependent kinetics and no global conservation laws",
{\it SIAM J. Appl. Math.}, 71, pp. 128--146, 2011.

\bibitem{Bollobas}
B. Bollobas, {\it Modern Graph Theory}, Graduate Texts in Mathematics 184, Springer, New York, 1998.

\bibitem{CDC}
A. Chapman and M. Mesbahi, ``Advection on graphs", {\it $50^{\tiny{\text{th}}}$ IEEE CDC-ECC}, Orlando, USA, pp. 1461--1466, 2011.

\bibitem{Toric}
G. Craciun, A. Dickenstein, A. Shiu, B. Sturmfels, ``Toric dynamical systems", {\it Journal of Symbolic Computation}, 44, pp. 1551--1565, 2009.

\bibitem{Craciun}
G. Craciun, M. Feinberg, ``Multiple equilibria in complex chemical reaction networks: I. the injectivity property", {\it SIAM J. Appl. Math.}, 65(5), pp. 1526--1546, 2006.

\bibitem{Dick}
A. Dickenstein and M.P. Mill$\acute{\text{a}}$n, ``How far is complex balancing from detailed balancing", {\it Bull. Math. Biol.}, 73, pp. 811--828, 2011.

\bibitem{Doerfler}
F. D\"orfler, F. Bullo, "Kron Reduction of Graphs with Applications to Electrical Networks", {\it IEEE Transactions on Circuits and Systems I}, 99, pp. 1-14, 2012.

\bibitem{Feinberg2}
M. Feinberg, 
``Complex balancing in chemical kinetics", {\it Arch. Rational Mech. Anal.}, 49, pp. 187--194, 1972.

\bibitem{Feinberg}
M. Feinberg, 
``Chemical reaction network structure and the stability of complex isothermal reactors -I. The deficiency zero and deficiency one theorems",
{\it Chemical Engineering Science}, 43(10), pp. 2229--2268, 1987.

\bibitem{Feinberg1}
M. Feinberg,
``The existence and uniqueness of steady states for a class of chemical reaction networks",
{\it Arch. Rational Mech. Anal.}, 132, pp. 311--370, 1995.

\bibitem{Hardin2010}
H.M. H\"ardin, {\it Handling Biological Complexity: As simple as possible but not simpler}, Ph.D. Thesis, Vrije Universiteit Amsterdam, 2010.

\bibitem{Horn}
F.J.M. Horn,
``Necessary and sufficient conditions for complex balancing in chemical kinetics",
{\it Arch. Rational Mech. Anal.}, 49, pp. 172--186, 1972.

\bibitem{HornJackson}
F. Horn and R. Jackson,
``General mass action kinetics",
{\it Arch. Rational Mech. Anal.}, 47, pp. 81--116, 1972. 

\bibitem{OursMTNS}
B. Jayawardhana, S. Rao, A. van der Schaft, ``Balanced chemical reaction networks governed by general kinetics", Proc. 20th Mathematical Theory of Networks and Systems, Melbourne, June 2012.


\bibitem{Kron}
G. Kron, {\it Tensor Analysis of Networks}, Wiley, 1939. 

\bibitem{McKeithan}
T.W. McKeithan, ``Kinetic proofreading in T-cell receptor signal transduction", {\it Proc. Natl. Acad. Sci. USA}, 92, pp. 5042-5046, 1995.

\bibitem{Niezink}
N.M.D. Niezink, ``Consensus in networked multi-agent systems", Master's thesis in Applied Mathematics, Faculty of Mathematics and Natural Sciences, University of Groningen, August 2011.

\bibitem{Othmer}
H. G. Othmer, {\it Analysis of Complex Reaction Networks}, Lecture Notes, School of Mathematics, University of Minnesota, December 9, 2003.

\bibitem{Palsson}
B.O. Palsson, {\it Systems Biology; Properties of Reconstructed Networks}, Cambridge University Press, Cambridge 2006.

\bibitem{Prescott2012}
T.P. Prescott, A. Papachristodoulou, ``Guaranteed error bounds for structured complexity reduction of biochemical networks", Journal of Theoretical Biology, vol. 304, pp. 172-182, 2012.

\bibitem{OursACC}
S. Rao, B. Jayawardhana, A.J. van der Schaft, ``On the graph and systems analysis of reversible chemical reaction networks with mass action kinetics", Proc. IEEE American Control Conference, Montreal, June 2012.

\bibitem{Siegel}
D. Siegel and D. MacLean, ``Global stability of complex balanced mechanisms", {\it Journal of Mathematical Chemistry}, 27, pp. 89--110, 2000.

%\bibitem{Siegel1}
%D. Siegel and M.D. Johnston, ``A stratum approach to global stability of complex balanced systems", http://arxiv.org/abs/1008.1622v2 (2010).

%\bibitem{Othmer}
%H. G. Othmer, {\it Analysis of Complex Reaction Networks}, Lecture Notes, School of Mathematics, University of Minnesota, December 9, 2003.

\bibitem{Sunnaker2011} 
M. Sunn\aa ker, G. Cedersund, M. Jirstrand, ``A method for zooming of nonlinear models of biochemical systems", BMC Systems Biology, 5:140, 2011.


\bibitem{Sontag}
E.D. Sontag, ``Structure and stability of certain chemical networks and applications to the kinetic proofreading model of T-cell receptor signal transduction", {\it IEEE Trans. Autom. Control}, 46(7), pp. 1028--1047, 2001.

\bibitem{vdsSCL}
A.J. van der Schaft,
``Characterization and partial synthesis of the behavior of resistive circuits at their terminals",
{\it Systems \& Control Letters}, 59, pp. 423--428, 2010.

\bibitem{Ours}
A.J. van der Schaft, S. Rao and B. Jayawardhana, ``On the Mathematical Structure of Balanced Chemical Reaction Networks Governed by Mass Action Kinetics", http://arxiv.org/abs/1110.6078 (2011).

\end{thebibliography}
\end{document}